\documentclass[11pt,a4paper]{article}

\usepackage{amsmath,amsthm,amssymb,a4wide}
\usepackage{times}
\usepackage{enumerate}

\usepackage{setspace}

\usepackage{dsfont}

\DeclareMathOperator{\tr}{tr}

\newcommand{\R}{\mathbb R}

\newcommand{\gep}{\varepsilon}
\newcommand{\RR}{\mathbb{R}}
\newcommand{\RRn}{\mathbb{R}^n}

\newcommand{\open}{(\Omega_t)_{t\in(0,T)}}

\newcommand{\close}{(\mathcal{F}_t)_{t\in(0,T)}}

\providecommand{\norm}[1]{\lVert#1\rVert}
\providecommand{\abs}[1]{\lvert#1\rvert}

\theoremstyle{definition}
\newtheorem{thm}{Theorem}[section]
\newtheorem{cor}[thm]{Corollary}
\newtheorem{prop}[thm]{Proposition}
\newtheorem{lem}[thm]{Lemma}
\newtheorem{defin}[thm]{Definition}
\newtheorem{rem}[thm]{Remark}

\newtheorem{exa}[thm]{Example}

\newcommand{\subjclass}[1]{\bigskip\noindent\emph{2010 Mathematics Subject Classification:}\enspace#1}

\numberwithin{equation}{section}

\begin{document}

\title{On viscosity and equivalent notions of solutions\\ for anisotropic geometric equations}
\author{Cecilia De Zan \& Pierpaolo Soravia\thanks{email: soravia@math.unipd.it.  }\\
Dipartimento di Matematica\\ Universit\`{a} di Padova, via Trieste 63, 35121 Padova, Italy}

\date{}
\maketitle

\begin{abstract}
We prove that viscosity solutions of geometric equations in step two Carnot groups can be equivalently reformulated by restricting the set of test functions at the singular points. These are characteristic points for the level sets of the solutions and are usually difficult to deal with. 
A similar property is known in the euclidian space, and in Carnot groups is based on appropriate properties of a suitable homogeneous norm.
We also use this idea to extend to Carnot groups the definition of generalised flow, and it works similarly to the euclidian setting.
These results simplify the handling of the singularities of the equation, for instance to study the asymptotic behaviour of singular limits of reaction diffusion equations. 
We provide examples of using the simplified definition, showing for instance that boundaries of strictly convex subsets in the Carnot group structure become extinct in finite time when subject to the horizontal mean curvature flow even if characteristic points are present.
\end{abstract}


\section{Introduction}

In this paper we want to discuss the notion of viscosity solution for geometric equations,
describing weak front propagation in step two Carnot groups, of the form
\begin{equation}\label{eqflow}
u_t(x,t)+F(x,t,Xu(x,t),X^2u(x,t))=0,\quad(x,t)\in\R^n\times(0,+\infty).
\end{equation}
Here the operator $F=F(x,t,q,A)$, $F:\R^n\times(0,+\infty)\times\R^m\backslash\{0\}\times{\mathcal S}^m\to\R$ is elliptic and {\it geometric}, meaning that it is positively one homogeneous in the pair $(q,A)\in\R^m\backslash\{0\}\times{\mathcal S}^m$ and invariant in the last argument with respect to matrices of the form $\mu\; q\otimes q$, $\mu\in\R$, as we make it more precise later. The notation ${\mathcal S}^m$ indicates the set of symmetric $m\times m$ matrices, $n,\;m\geq2$. Therefore it is possible that $F$ has a singularity at $q=0$ and we assume that it behaves nicely, namely
$$F_*(x,t,0,\mathbb{O})=F^*(x,t,0,\mathbb{O})=0,$$
where the stars above indicate the lower and upper semicontinuous envelopes, respectively. The notation $Xu$ indicates the horizontal gradient with respect to a family of vector fields $\{X_1,\dots,X_m\}$, seen as differential operators,
\begin{equation}\label{eqfields}
X_j=\sum_{i=1}^n\sigma_{i,j}(x)\partial_i,\quad j\in\{1,\dots,m\},
\end{equation}
generators of a step two Carnot group. In particular, for a smooth function $u$, $Xu=\nabla u\;\sigma(x)$, $\sigma=(\sigma_{i,j})_{i,j}$ and if $m<n$ the equation (\ref{eqflow}) has singularities when $Xu=0$, i.e. at characteristic points of the level set of $u$, therefore on a subspace of positive dimension.
Notation $X^2u$ indicates instead the horizontal hessian, namely $X^2u=\left(X_iX_ju\right)^*_{i,j=1,\dots,m}$, the symmetrised matrix of second derivatives.
This compares to the usual euclidian case when $\sigma\equiv I_{n}$ the identity matrix, where $Xu\equiv \nabla u$ is the standard gradient, and the singularity is just at the origin. 
In the special case when the operator $F:\RR^m\backslash\{0\}\times\mathcal{S}^m\to\RR$ is defined as
\begin{equation}\label{mean curvature hamiltonian dim m}
F(q,A)=-\tr[\big(I-\frac{q}{\abs{q}}\otimes\frac{q}{\abs{q}}\big)A],
\end{equation}
and moreover $m=n$ and $\sigma\equiv I_n$, (\ref{eqflow}) reads as the well known the mean curvature flow equation
\begin{equation}\label{eqmeanflow}
u_t(x,t)-\tr[\big(I-\frac{\nabla u}{\abs{\nabla u}}\otimes\frac{\nabla u}{\abs{\nabla u}}\big)D^2u]=0.
\end{equation}
In a group setting instead, (\ref{eqmeanflow}) becomes
\begin{equation}\label{eqlevelset(vector fieldX)}
u_t(x,t)-\sum_{i,j=1}^m\Big(\delta_{ij}-\frac{X_iu(x,t)X_ju(x,t)}{\sum_{i=1}^m(X_iu(x,t))^2}\Big)X_iX_ju(x,t)=0,
\end{equation}
which is the horizontal mean curvature flow equation in the Carnot group.

Due to the presence of singularities and the fact that we do not expect classical solutions in general in (\ref{eqflow}), we will use as usual the notion of viscosity solution, as in Crandall, Ishii, Lions \cite{cil}, Chen, Giga, Goto \cite{cgg}. In our main result, we prove an equivalent notion of solution where we use a restricted class of test functions at singular points, with the property that if the horizontal gradient vanishes, then also the horizontal hessian vanishes as well.
This equivalent notion of solution simplifies the dealing with singularities and was first proved in the euclidian setting for the mean curvature flow equation by Barles, Georgelin \cite{bage} to study the convergence of numerical schemes. 
We also use this approach to extend to our setting the notion of generalised flow, introduced as a general and flexible method to study singular limits in pdes giving rise to propagating fronts by Barles and Souganidis \cite{basou} and applied in several situations in the euclidian setting, see also Barles and Da Lio \cite{badl}.
As a matter of fact, we will use this notion of solution in a forthcoming paper, when we discuss the singular limit of reaction diffusion equations for anisotropic and degenerate diffusions \cite{deso5}, while we develop here the preliminary needed tools on weak front propagation. This simplified approach, which is particularly helpful when studying approximations of (\ref{eqflow}) of different nature, therefore extends to the Carnot group setting with similar properties. Hopefully it could also prove useful to tackle the comparison principle for viscosity solutions of (\ref{eqflow}), which is still missing in the literature in full generality.
To achieve our goal we need to modify the usual approach with the doubling of variables in viscosity solutions, by changing the test function, since the euclidian norm does not work for singular anisotropic equations as (\ref{eqflow}), and replace it instead with an homogeneous norm, adapted to the Carnot group structure.
As an application, we show how one can more easily check that functions are super or subsolutions of (\ref{eqflow}) especially at singular points, by providing explicit examples of super or subsolutions to be used as barriers. If in particular we consider the recent notion of v-convex functions with respect to the family of vector fields, we can prove, coupling our result with a comparison principle, that their level sets become extinct in finite time under the horizontal mean curvature flow equation, by constructing suitable supersolutions of (\ref{eqflow}).

Equation (\ref{eqflow}) appears in the level set approach to the weak propagation of hypersurfaces, where we want to discuss the propagation of interfaces, boundaries of open sets, with prescribed normal velocity. In the euclidian space usually the velocity $V=V(x,n,Dn)$, where $n$ is the exterior normal. Indeed, if $\Omega_t\subset\R^n$ is a family of open sets, $\Gamma_t=\partial\Omega_t$ is the propagating front, and there exists a smooth function $u:\RRn\times[0,+\infty)\to\RR$ such that
$$\Gamma_t=\{x\in\RRn:u(x,t)=0\},\quad\Omega_t=\{x\in\RRn:u(x,t)>0\},\quad \nabla u\neq0\mbox{ on }\Gamma_t$$
then one computes
$$V=\frac{u_t}{|\nabla u|},\quad\mathbf n=-\frac{\nabla u}{|\nabla u|}\quad\mbox{and}\quad D\mathbf n=-\frac{1}{|\nabla u|}\Big(I-\frac{\nabla u\otimes \nabla u}{|\nabla u|^2}\Big)D^2u$$
and so $u$ formally satisfies
\begin{equation}\label{level set equation}u_t=G(x,t,\nabla u,D^2u),\end{equation}
where $G$ is related to $V$ by
$$G(x,t,p,A)=|p|V\Big(x,t,-\frac p{|p|},-\frac1{|p|}(I-\frac{p\otimes p}{|p|^2})A\Big),\qquad(x,p,A)\in\RRn\times\R^n\backslash\{0\}\times\mathcal S^n.$$
In our case the anisotropy of the velocity will be for instance exploited by the fact that
$$G(x,t,p,A)=F(p\sigma(x),\;^t\sigma(x)A\sigma(x)),$$
so that as an operator $G(x,t,\nabla u,D^2u)=F(Xu,X^2u)$.
The novelty here with respect to the classical cases is that while in the euclidian case $\sigma=I$, and its square is a non degenerate matrix, here the diffusion matrix $\sigma(x)^t\sigma(x)$ is not only anisotropic but also degenerate. When the family of vector fields does not span the whole $\R^n$ at each point, this fact adds metric singularities to the usual one of geometric equations.

The geometric property of the level set approach is based on the fact that if $u$ solves (\ref{eqflow})
and $\psi:\R\to\R$ is smooth and increasing, then also $\psi(u)$ solves the same equation. As a consequence, when a comparison principle holds true, it is easy to see that if $u^1_o$ and $u^2_o$ are two initial conditions such that
$$\Gamma_o=\{x:u^1_o(x)=0\}=\{x:u^2_o(x)=0\},$$
and $u^1, u^2$ are the corresponding solutions in (\ref{eqflow}), then one has
\begin{equation*}
{\{x:u^1(x,t)=0\}=\Gamma_t=\{x:u^2(x,t)=0\}},\quad \mbox{for all }t>0.
\end{equation*}
One can therefore { define} the family of closed sets $(\Gamma_t)_t$ to be the geometric flow of the front or interface $\Gamma_o$ with the prescribed normal velocity.

The notion of horizontal normal and horizontal mean curvature is due to Danielli, Garofalo, Nhieu \cite{dagani}. Recently equation (\ref{eqflow}) has been studied by several authors. Existence results are available in the work of Capogna, Citti \cite{caci}, who proved existence in Carnot groups by vanishing viscosity riemannian approximations. Dirr, Dragoni, Von Renesse \cite{didr} used stochastic approximations to show existence for more general H\"ormander structures. Capogna, Citti, Manfredini \cite{cacima} prove uniform regularity estimates on the riemannian vanishing viscosity approximations for the flow of graphs, that also apply to prove existence for (\ref{eqflow}) in that case. On uniqueness results the literature is far less complete. Capogna, Citti \cite{caci} proved a comparison principle if either one of the functions to compare is uniformly continuous or their initial condition does not depend on the vertical coordinate, thus avoiding characteristic points in the initial front. A very recent paper by Baspinar, Citti \cite{baci} finds a comparison principle in Carnot groups of step two as a consequence of the fact that all solutions are limits of suitable families of riemannian regularisations. We remark the fact that in \cite{caci}, \cite{didr} the authors use a notion of solution that differs from standard viscosity solutions at singular points. However their notion of solution turns out to be equivalent to viscosity solutions as a consequence of our result.
One of the referees pointed out to us the work of Ferrari, Liu and Manfredi \cite{flm} where the authors use an approach similar to ours in the case of the horizontal mean curvature flow equation in the Heisenberg group, and they show a comparison principle for axisymmetric viscosity solutions.

We recall that the level set method for geometric flows was proposed by Osher-Sethian \cite{os} for numerical computations of geometric flows.
The rigorous theory of weak front evolution started with the work by Evans-Spruck \cite{es} for the mean curvature flow and by Chen-Giga-Goto \cite{cgg} for more general geometric flows.
For the mathematical analysis of the level set method via viscosity solutions, the reader is referred to the book by Giga \cite{gi}, where the approach is discussed in detail, see also Souganidis \cite{soug} and the references therein for the main applications of the theory.


\section{Step two Carnot groups and level set equations on the group}

In this paper we consider in $\R^n$ a family of vector fields
${\mathcal X}=\{X_1,\dots,X_m\}$ written as differential operators as in (\ref{eqfields}) and consider
$\sigma:\R^n\to\R^{n\times m}$ which is the matrix valued family of the coefficients. We will indicate $\sigma_j$ the $j-$th column of $\sigma$ so that
$$X_j\;I(x)=\sigma_j(x),\quad j\in\{1,\dots,m\}$$
where $I(x)$ is the identity map in $\R^n$ and in general $X_j$ applied to a vector valued smooth function $\varphi$ means the vector whose entries are given by $X_j$ applied to the components of $\varphi$.
The vector fields of the family are throughout the paper assumed to be generators of a step two Carnot group.
To be more precise we rely on the following definition, see the book by Bonfiglioli, Lanconelli, Uguzzoni \cite{bolaug}, which we refer the reader to, for an introduction to the subject.
\begin{defin}
We say that $G=(\R^n,\circ)$ is a Lie group if $\circ$ is a group operation on $\R^n$ and the map $(x,y)\mapsto x^{-1}\circ y$ is smooth. 

\noindent
We then say that 
$(G,\circ,\delta_\lambda)$ is a step two Carnot group if we can split $\R^n=\R^m\times\R^{n-m}$, $x=(x_h,x_v)$, $m<n$, and for all $\lambda>0$ the family of dilations $\delta_\lambda(x)=(\lambda x_h,\lambda^2 x_v)$ are automorphisms of the group (the group is homogeneous). Moreover the family of vector fields $\mathcal X$ are left invariant on $G$ with respect to the group operation, that is for all $\varphi\in C^{\infty}(\R^n)$ and all $\alpha\in\R^n$ we have that
$$X_j(\varphi(\alpha\circ x))=(X_j\varphi)(\tau_\alpha(x)),\quad j\in\{1,\dots,m\},$$
where $\tau_\alpha(x)=\alpha\circ x$ is the left traslation, and the following H\"ormander property is satisfied
$$\hbox{span}\{X_i(x),[X_j,X_k](x):i,j,k\in\{1,\dots,m\}\}=\R^n,\quad\hbox{for all }x\in\R^n,$$
so that the family of vector fields $\mathcal X$, together with their first order Lie brackets, generates $\R^n$ at every point (the Carnot group is step two).

\noindent
The vector fields of the family $\mathcal X$ are said to be generators of the Carnot group.
\end{defin}

Following \cite{bolaug}, it is then well known that if $\mathcal X$ generate a step two Carnot group, then, by a suitable change of variables, we can suppose that
\begin{equation}\label{eqcarnotstr}
\sigma(x)=\left(\begin{array}{c}
I_m\\^t(Bx_h)\end{array}\right),
\end{equation}
where $I_m$ is the $m\times m$ identity matrix, $Bx_h=(B^{(1)}x_h,\dots,B^{(n-m)}x_h)$, and $B^{(j)}$, $j\in\{1,\dots,n-m\}$ are skew symmetric, linearly independent, $m\times m$ matrices.
In addition, $\R^n$ has the group structure with the operation
$$x\circ y=(x_h+y_h,x_v+y_v+<Bx_h,y_h>),$$
with the notation $<Bx_h,y_h>=(B^{(1)}x_h\cdot y_h,\dots,B^{(n-m)}x_h\cdot y_h)$. With this group operation it is clear that $x^{-1}=-x$ and $0$ is the identity element of the group.

Moreover we notice that the jacobian of the left traslation has the following structure
$$D\tau_\alpha(x)=\left(\begin{array}{cc}
I_m\quad &{\mathbb O}_{m\times (n-m)}\\^t(Bx_h)&I_{n-m}
\end{array}\right),$$
so the first $m$ columns of the jacobian give the matrix $\sigma(x)$.
It is also good to remember that for $\lambda>0$ the family $\mathcal X$ is homogeneous of degree one with respect to the dilations, namely
$$X_j(\varphi(\delta_\lambda(x)))=\lambda(X\varphi)(\delta_\lambda(x)),\quad j\in\{1,\dots,n-m\},$$
for all $\varphi\in C^{\infty}(\R^n)$.

\begin{exa}
The well known example of the Heisenberg group comes from $\R^3=\R^2\times\R$ and the single matrix $$B=\left(\begin{array}{cc}0\quad&1\\-1&0\end{array}\right).$$
\end{exa}

For our purposes, given a smooth function $u\in C^2(\R^n)$ we indicate the {\it horizontal gradient} (here gradients are row vectors) as
$$Xu(x)=\nabla u(x)\;\sigma(x),$$
and the {\it horizontal hessian} as
$$X^2u(x)=\left(X_jX_k\;u(x)\right)^*_{j,k=1,\dots,n-m}=\;^t\sigma(x)D^2u(x)\sigma(x).$$
We just observe that $A^*=(A+\;^tA)/2$ indicates the symmetrisation and that the first order terms in the second derivatives of $X^2$ cancel out by direct computation since $\sigma$ only depends on the first $m$ variables.

In $\R^n$, taking advantage of the group structure of the family of vector fields, we want to study the problem of weak front propagation by extending the now classical level set idea.
Let $F:\R^n\times(0,+\infty)\times\R^m\backslash\{0\}\times\mathcal S^m\to\RR$ be a continuous function, locally bounded at points of the form $(x,t,0,A)$, where $\mathcal S^m$ denotes the space of the $m\times m$ symmetric matrices. We assume on $F$ the following structure conditions.
\begin{description}
  \item[(F1)] $F$ satisfies
      \begin{equation}\label{0 condition}
      F^*(x,t,0,\mathbb{O})=F_*(x,t,0,\mathbb{O}),\qquad\mbox{for all }(x,t)\in\R^n\times(0,+\infty);
      \end{equation}
  \item[(F2)] $F$ is elliptic, i.e. for any $(x,t)\in\R^n\times(0,+\infty),\;p\in\R^m\backslash\{0\}$ and $A,B\in\mathcal S^m$
  \begin{equation}\label{ellipticity condition}
  F(x,t,p,A)\leq F(x,t,p,B),\quad\mbox{if }A\geq B;
  \end{equation}
  \item[(F3)] $F$ is \emph{geometric}, i.e.,
 \begin{equation}\label{geometric condition}
 F(x,t,\lambda p,\lambda A+\mu(p\otimes p))=\lambda F(x,t,p,A)\quad\mbox{for all }\lambda>0\mbox{ and }\mu\in\RR
 \end{equation}
for every $(x,t)\in\R^n\times(0,+\infty),\;p\in\R^m\backslash\{0\}$ and $A\in\mathcal S^m$.
\end{description}
In the above, we are using the following notation for the lower semicontinuous extension of $F$ at the singular points.
$$F_*(x,t,0,A)=\lim_{r\to0+}\inf\{F(y,t,q,B):q\neq0,\;|(y,q,B)-(x,0,A)|\leq r\},$$
and similarly for the upper semicontinuous extension $F^*$. Notice in particular that the geometric property of $F$ implies $F_*(x,t,0,\mathbb{O})=0$ for all $(x,t)\in\R^n\times(0,+\infty)$.

We want to discuss the notion of solution for the equation
\begin{equation}\label{eqlevelset}
u_t(x,t)+F(x,t,Xu,X^2u),\quad(x,t)\in\R^n\times(0,+\infty),\end{equation}
where now only the horizontal first and second derivatives of the unknown function appear in the equation.
Notice that in our group setting, the operator $F$ in (\ref{eqlevelset}), written in the usual coordinates of $\R^n$ becomes
\begin{equation}\label{eqopg}
G(x,t,p,A)=F(x,t,p\sigma(x),\;^t\sigma(x)A\sigma(x)),
\end{equation}
$G:((\R^n\times(0,+\infty)\times\R^n)\backslash\{(x,t,p):p\sigma(x)=0\})\times {\mathcal S}^m\to\R.$

\begin{rem}
We easily show in a moment that $G$ preserves the assumptions (F1), (F2), (F3), however the singularities of $G$ are not just at the origin but in the whole of the subset
$$S=\{(x,t,p,A)\in\R^n\times(0,+\infty)\times\R^m\times{\mathcal S}^m:p\sigma(x)=0\},$$
where now for all $(x,t,A)\in\R^n\times(0,+\infty)\times{\mathcal S}^m$, the set $\{p:(x,t,p,A)\in S\}$ is a varying subspace, not necessarily trivial if the family of vector fields $\mathcal X$ does not span $\R^n$ at $x$. In this sense the operator $G$ is not covered by the standard theory of the anisotropic operators.\\
Operator $G$ is elliptic since if $A\geq B$, then $^t\sigma(x)A\sigma(x)\geq \;^t\sigma(x)B\sigma(x)$ and thus $G(x,t,p,A)\leq G(x,t,p,B)$.\\
Operator $G$ is also geometric since
$$\begin{array}{l}
G(x,t,\lambda p,\lambda A+\mu(p\otimes p))=F(x,t,\lambda p\sigma(x),\lambda\;^t\sigma(x) A\sigma(x)+\mu(p\sigma(x)\otimes p\sigma(x)))\\
=\lambda F(x,t, p\sigma(x),\;^t\sigma(x) A\sigma(x))=\lambda G(x,t,p,A).
\end{array}$$
Thus (F2), (F3) hold true.
\end{rem}
We now recall the usual definition of viscosity solution for the level set equation (\ref{eqlevelset}).
\begin{defin}\label{defviscsol}
An upper (respectively lower) semicontinuous function $u:\RRn\times(0,+\infty)\to\R$ is a viscosity subsolution  (respectively supersolution) of (\ref{eqlevelset}) if and only if for any $\phi\in C^2(\RRn\times(0,+\infty))$, if $(x,t)\in\RRn\times(0,+\infty)$ is a local maximum (respectively minimum) point for $u-\phi$, we have
\begin{equation}\label{eqsubsol}\phi_t(x,t)+G_*\big(x,t,\nabla \phi(x,t),D^2\phi(x,t)\big)
\leq0,
\end{equation}
where $G$ is given in (\ref{eqopg}).
A viscosity solution of (\ref{eqlevelset}) is a continuous function $u:\RRn\times(0,+\infty)\to\R$ which is either a subsolution and a supersolution.
\end{defin}
\begin{rem}\label{remtech}
In the previous definition the lower semicontinuous extension of $G$ at the singular points where $p\sigma(x)=0$ is
$$\begin{array}{ll}
G_*(x,t,p,A)&\\
=\lim_{r\to0+}\inf\{G(y,s,q,B):(y,s,q,B)\in\hbox{dom }(G),\;|(x-y,t-s,p-q,A-B)|\leq r\}\\
=\lim_{r\to0+}\inf\{F(y,s,q\sigma(y),\;^t\sigma(y)B\sigma(y)):q\sigma(y)\neq0,
\quad|(x-y,t-s,p-q,A-B)|\leq r\}.
\end{array}$$
In particular from $p\sigma(x)=0$ we have
$$\begin{array}{ll}
F_*(x,t,0,\;^t\sigma(x)A\sigma(x))\leq G_*(x,t,p,A)
\leq G^*(x,t,p,A)\leq F^*(x,t,0,\;^t\sigma(x)A\sigma(x)).
\end{array}$$
Thus if $p\sigma(x)=0$ and $^t\sigma(x)A\sigma(x)={\mathbb O}$, then $G_*(x,t,p,A)=G^*(x,t,p,A)=0$, so a counterpart of (F1) holds for $G$.\\
In Definition \ref{defviscsol}, if $X\phi(x,t)\neq0$, then (\ref{eqsubsol}) is equivalently written as
\begin{equation}
\phi_t(x,t)+F\big(x,t,X\phi(x,t),X^2\phi(x,t)\big)\leq0,
\end{equation}
and the extended operator $G_*$ only appears when $X\phi(x,t)=0$.
Therefore at singular points the notion of viscosity subsolution is stronger than one would get requiring
\begin{equation}\label{eqfakesubsol}
\phi_t(x,t)+F_*\big(x,t,X\phi(x,t),X^2\phi(x,t)\big)
\leq0.
\end{equation}
instead of (\ref{eqsubsol}).
Notice that in the special case (\ref{mean curvature hamiltonian dim m}), if $p\sigma(x)=0$,
$$F_*(x,t,0,A)=\min_{|p|=1}\{-\hbox{tr }(I-p\otimes p)A\}
$$
and this is used in \cite{caci} or in \cite{didr} to define (weak-)subsolutions of the horizontal mean curvature flow equation, by requiring (\ref{eqfakesubsol}) instead of (\ref{eqsubsol}).
\end{rem}

\section{Viscosity solutions}

In this section we consider equation (\ref{eqlevelset}) and prove an equivalent definition of viscosity solution. This result extends \cite{bage} to our setting and simplifies the treatment of singularities of equation (\ref{eqlevelset}) by restricting the family of test functions at characteristic points.

When it will be necessary to emphasise the variable $x$ in which we are computing the vector fields $X_i$ (and with respect to we are computing the derivatives), we will denote the horizontal gradient and the horizontal Hessian matrix as $X_x$ and $X_x^2$ .  For example if $H(x,y)$ is a $C^2$ function defined in $\R^n\times\R^n$ and $(x_o,y_o)$ is a generic point of $\R^n\times\R^n$ we will denote with $X_xH(x_o,y_o)$ the horizontal gradient of $H$ with respect to the variable $x$ and with $X_yH(x_o,y_o)$ the horizontal gradient of $H$ with respect to $y$, both computed at the point $(x_o,y_o)$. Analogous definitions hold for $X^2_{x}H(x_o,y_o)$ and $X^2_{y}H(x_o,y_o)$.
We consider an homogeneous (with respect to any dilatation $\delta_\lambda$, $\lambda>0$) norm on $\R^n$,
\begin{equation}\label{norm}\norm{x}_{G}=[|x_h|^4+|x_v|^2]^{1/4},\end{equation}
and we define a left invariant metric $d_{G}:\R^n\times\R^n\to[0,+\infty)$ as
\begin{equation}\label{distance}
d_{G}(x,y)=\norm{x^{-1}\circ y}_G
=[\abs{y_h-x_h}^4+\abs{y_v-x_v-\langle Bx_h,y_h\rangle}^2]^{1/4}.
\end{equation}
\begin{rem}
Here we make some comments on the definitions (\ref{norm}) and (\ref{distance}). Dealing with fully nonlinear partial differential equations with singularities poses a number of additional difficulties. Viscosity solutions theory can cope with these difficulties since the work of Evans-Spruck \cite{es} and Chen-Giga-Goto \cite{cgg}. The horizontal mean curvature flow equation adds further difficulties since the singularity does not just appear when the gradient of the solution vanishes, but rather when the horizontal gradient vanishes, so when the gradient takes its values in a nontrivial subspace. In some key step of the proofs, the standard euclidian distance does not work and one has to think to something different. One natural choice would be to exchange the euclidian distance with the Carnot-Caratheodory distance. This distance is not smooth however being only locally H\"older continuous. Therefore, due to the nature of Carnot groups, one thinks of distance functions that are related to homogeneous norms, which are distance function equivalents to the euclidian one but are smooth. One well known example is the norm in (\ref{norm}). This one works well in step two groups, at least for the results we prove, but not for the comparison principle, one reason being that the group operation is not commutative and this makes the distance not symmetric. In groups of higher step, one has a natural homogeneous distance with more terms, making the computations in this section more complex. Moreover we are often using the structure (\ref{eqcarnotstr}), which is valid specifically in step two groups. There might be additional difficulties due to the fact that Carnot groups with step higher than two differ in some important geometric properties. Nonetheless step two groups already have important applications that make their study quite interesting as for instance in models of the visual cortex, see \cite{baci} and the references therein for details.
\end{rem}

We start proving a nice property of the homogeneous metric $d_G$ defined in (\ref{distance}).
 \begin{lem}\label{lemma:property norm in Carnot group2} Put $N(x)=\norm{x}_{G}^4$ for any $x\in\RR^{n}$. Then
 \begin{enumerate}[(i)]
\item $\{x\in \R^n:|XN(x)|=0\}=\{x\in \R^n:X^2N(x)=\mathbb{O}\}=\{x\in \R^n:x_h=0\}$.
\item  $|X_xd_{G}^4(x,y)|=|X_yd_{G}^4(x,y)|$ and $X_x^2d_{G}^4(x,y)=X_y^2d_{G}^4(x,y)$ for any $x,y\in \R^n$;
 moreover they all have as zero-set the set $\{(x,y)\in \R^n\times\R^n:x_h=y_h\}$.
\end{enumerate}
\end{lem}
\begin{proof}(i) The proof of the first point follows by some simple computations. In fact since
$$XN(x)=4\abs{x_h}^2x_h+2\sum_{k=1}^{n-m}{(x_v)}_kB^{(k)}x_h,$$
we have, here notice that, since the matrices $B^{(k)}$ are all skew symmetric the mixed products are all null,
$$|XN(x)|^2=16|x_h|^6+4\left|\sum_{k=1}^{n-m}{(x_v)}_kB^{(k)}x_h\right|^2=16|x_h|^6+4\sum_{k,l=1}^{n-m}(x_v)_k(x_v)_l\langle B^{(k)}x_h,B^{(l)}x_h\rangle.$$
Thus $XN(x)=0$ if and only if $x_h=0$. Moreover
$$X^2N(x)=4|x_h|^2I_m+8x_h\otimes x_h+2\sum_{k=1}^{n-m}B^{(k)} x_h\otimes B^{(k)}x_h,$$
which is null at $x_h=0$.

\noindent
(ii) First of all we observe that, since the vector fields $X_i$ are invariant by left composition of the group operation, we have
\begin{align}\label{eqdistnorm}
X_yd_{G}^4(x,y)&=X_yN(x^{-1}\circ y)=(XN)(x^{-1}\circ y)\\
X_y^2d_{G}^4(x,y)&=(X^2N)(x^{-1}\circ y)
\end{align}
and so by point (i) $X_yd_{G}^4(x,y)$ and $X^2_yd_{G}^4(x,y)$ are null if and only if $(x^{-1}\circ y)_h=0$, i.e. $y_h=x_h$. To compute the horizontal gradient and the horizontal Hessian matrix with respect the $x$ variable we observe that, since $N(x^{-1})=N(-x)=N(x)$, it holds $d_{G}^4(x,y)=N(x^{-1}\circ y)=N(y^{-1}\circ x)$ and, by left invariance of the vector fields,
$$X_xd_{G}^4(x,y)=(XN)(y^{-1}\circ x),\quad
X_x^2d_{G}^4(x,y)=(X^2N)(y^{-1}\circ x).$$
Again $X_xd_{G}^4(x,y)$ and $X^2_xd_{G}^4(x,y)$ are null exactly when $y_h=x_h$.

Finally we observe that $\abs{X_yd_{G}^4(x,y)}^2=\abs{X_xd_{G}^4(x,y)}^2$ and $X^2_yd_{G}^4(x,y)=X^2_xd_{G}^4(x,y)$.
 \end{proof}
We use the previous Lemma to prove an equivalent definition of solution other than Definition \ref{defviscsol} which is the usual definition of {viscosity solution} for the equation (\ref{eqlevelset}). The definition will only change at singular points of the differential operator.

\begin{thm}\label{thmbg} An upper (respectively lower) semicontinuous function $u$ is a viscosity subsolution  (respectively supersolution) of (\ref{eqlevelset}) if and only if for any $\phi\in C^2(\RRn\times(0,+\infty))$, if $(x,t)\in\RRn\times(0,+\infty)$ is a local maximum (respectively minimum) point for $u-\phi$, one has
\begin{equation}\label{condition1}\frac{\partial\phi(x,t)}{\partial t}+F(x,t,X\phi(x,t),X^2\phi(x,t))
\leq0\quad\mbox{if }X\phi(x,t)\neq0 \end{equation}
and
\begin{equation}\label{condition2}\frac{\partial\phi(x,t)}{\partial t}\leq0\quad\mbox{if}\quad X\phi(x,t)=0\;\mbox{and}\;\mbox X^2\phi(x,t)=0,\end{equation}
(respectively
$$\frac{\partial\phi(x,t)}{\partial t}+F(x,t,X\phi(x,t),X^2\phi(x,t)
\geq0\quad\mbox{if }X\phi(x,t)\neq0$$
and
\begin{equation}\label{condition2 supersol}
\frac{\partial\phi(x,t)}{\partial t}\geq0\quad\mbox{if}\quad X\phi(x,t)=0\;\mbox{and}\;\mbox X^2\phi(x,t)=0).
\end{equation}
\end{thm}
\begin{proof} We only show the result for subsolutions the other part being similar. It is clear that a viscosity subsolution will satisfy (\ref{condition2}) since $G_*(x,t,p,\mathbb{O})=0$ if $p\sigma(x)=0$ by Remark \ref{remtech} and (F1).
 
Let $u$ be an upper semicontinuous function which satisfies (\ref{condition1}) and (\ref{condition2}). Consider $\phi\in C^2(\RRn\times(0,+\infty))$ and $(\hat x,\hat t)\in\RRn\times(0,+\infty)$ a local maximum point for  $u-\phi$ such that $X\phi(\hat x,\hat t)=0$ and $X^2\phi(\hat x,\hat t)\neq0$. Without loss of generality we can assume that $u$ is a strict local maximum point for $u-\phi$. We need to prove that
\begin{equation}\label{thesis}
\frac{\partial\phi(\hat x,\hat t)}{\partial t}+G_*(x,t,\nabla\phi(\hat x,\hat t),D^2\phi(\hat x,\hat t))\leq0.
\end{equation}
For any $\gep>0$ we consider the function
$$\psi_\gep(x,y,t)=u(x,t)-\frac{d_{G}^4(x,y)}{\gep}-\phi(y,t).$$
By standard arguments one proves that for $\gep$ sufficiently small there is a family of local maxima $(x_\gep,y_\gep,t_\gep)$ of $\psi_\gep$ such that
$(x_\gep,y_\gep,t_\gep)$ converges to $(\hat x,\hat x,\hat t)$. 
Indeed, if $(x_\gep,y_\gep,t_\gep)$ are the maximum points of $\psi_\gep$ in a small compact neighborhood of $(\hat x,\hat x,\hat t)$, $(x_\gep,y_\gep,t_\gep)$ will converge to some $(\bar x,\bar y,\bar t)$ (passing to a subsequence if necessary). One first uses $\psi(x_\gep,y_\gep,t_\gep)\geq\psi(\hat x,\hat x, \hat t)$ to show that $\bar x=\bar y$ and next by taking the limit that $(\bar x,\bar t)$ is a maximum of $u-\phi$ in the neighborhood so that $(\bar x,\bar t)=(\hat x,\hat t)$.

Moreover since the function $y\mapsto\psi_\gep(x_\gep,y,t_\gep)$ has a local maximum in $y_\gep$ we have
$$\nabla \phi(y_\gep,t_\gep)=-\frac{D_yd_{G}^4(x_\gep,y_\gep)}{\gep},\quad
D^2\phi(y_\gep,t_\gep)\geq-\frac{D^2_yd_{G}^4(x_\gep,y_\gep)}{\gep}.$$
Thus
\begin{equation}\label{max consequence1}
X\phi(y_\gep,t_\gep)=-\frac{X_yd_{G}^4(x_\gep,y_\gep)}{\gep},\quad X^2\phi(y_\gep,t_\gep)\geq-\frac{X^2_yd_{G}^4(x_\gep,y_\gep)}{\gep},
\end{equation}
Two cases may now occur.

 1. $X\phi(y_\gep,t_\gep)=0$ along a subsequence. This means that $X_yd_{G}^4(x_\gep,y_\gep)=0$ and by Lemma \ref{lemma:property norm in Carnot group2}, $(x_\gep)_h=(y_\gep)_h$. Since the map $(x,t)\mapsto u(x,t)- \varphi(x,t)$, with $\varphi(x,t)=\frac{d^4_{G}(x,y_\gep)}{\gep}+\phi(y_\gep,t)$ attains a maximum at $(x_\gep,t_\gep)$ and
 $$X\varphi(x,t)=0\Leftrightarrow (x_\gep)_h=(y_\gep)_h\Leftrightarrow X^2\varphi(x,t)=0,$$
  by (\ref{condition2}) we get
    $$\frac{\partial\varphi}{\partial t}(x_\gep,t_\gep)=\partial_t\phi(y_\gep,t_\gep)\leq0.$$
    For future reference we remark that the test function $\varphi$ satisfies in a neighborhood of $(\hat x,\hat t)$: $X\varphi=0$ implies $X^2\varphi=0$.
    We proceed and by (\ref{max consequence1}) and $(x_\gep)_h=(y_\gep)_h$,
    we get that $X^2\phi(y_\gep,t_\gep)\geq\mathbb{O}$. Using the ellipticity of $F$ and Remark \ref{remtech}, it holds
    $$\begin{array}{l}
    \partial_t\phi(y_\gep,t_\gep)+G_*(y_\gep,t_\gep,{\nabla\phi(y_\gep,t_\gep)},D^2\phi(y_\gep,t_\gep))\leq\partial_t\phi(y_\gep,t_\gep)+F^*(y_\gep,t_\gep,X\phi(y_\gep,t_\gep),X^2\phi(y_\gep,t_\gep))\\
\leq\partial_t\phi(y_\gep,t_\gep)+F^*(y_\gep,t_\gep,0,\mathbb{O}_{m\times m})
    =\partial_t\phi(y_\gep,t_\gep)\leq0
    \end{array}$$
    and we conclude by letting $\gep$ go to 0.

2. $X\phi(y_\gep,t_\gep)\neq0$ for all $\gep$ sufficiently small. Using (\ref{max consequence1}) and the previous Lemma this means $(y_\gep)_h\neq (x_\gep)_h$.  Moreover the point $(x_\gep,t_\gep)$ is a maximum for 
$$\begin{array}{ll}
(x,t)\mapsto\psi_\gep(x,x\circ x_\gep^{-1}\circ y_\gep,t)&\displaystyle=u(x,t)-\frac{d^4_{G}(x_\gep,y_\gep)}{\gep}-\phi(x\circ x_\gep^{-1}\circ y_\gep,t)\\
&\displaystyle=:u(x,t)-\varphi(x,t),
\end{array}$$
since $d_G^4(x,x\circ x^{-1}_\gep\circ y_\gep)=N(x^{-1}_\gep\circ y_\gep)=d^4_G(x_\gep,y_\gep)$.
Let  $\tilde \tau_\alpha(x)=x\circ\alpha$ be the right translation by $\alpha$ and $D{\tilde\tau_\alpha}(x)\equiv D{\tilde\tau_\alpha}$ its Jacobian matrix. A simple computation shows that $D{\tilde\tau_\alpha}$ has the form
\begin{align*}D{\tilde\tau_\alpha}&=
\left(\begin{array}{c|c}
\mathbb{I}_m&\mathbb{O}_{m\times n}\\ \hline
\\
\;^t((\;^tB^{(1)})\alpha_h)&\\
\vdots&\mathbb{I}_n\\
\;^t((\;^tB^{(n-m)})\alpha_h)&\end{array}\right)=\left(\begin{array}{c|c}
\mathbb{I}_m&\mathbb{O}_{m\times n}\\ \hline
\\
\;^t(-B^{(1)}\alpha_h)&\\
\vdots&\mathbb{I}_n\\
\;^t(-B^{(n-m)}\alpha_h)&\end{array}\right)\\\\
&=\left(\begin{array}{c|c}
\mathbb{I}_m&\mathbb{O}_{m\times n}\\ \hline
\\
-\;^tB\alpha_h&\mathbb{I}_n,
\end{array}\right).\end{align*}
By the chain rule we get
$$\begin{array}{ll}
X\varphi(x_\gep,t_\gep)&=\;^t\sigma(x_\gep)\;^tD{\tilde\tau_{x_\gep^{-1}\circ y_\gep}}\nabla \phi(\tilde\tau_{x_\gep^{-1}\circ y_\gep}(x_\gep),t_\gep)
=\;^t\big(D{\tilde\tau_{x_\gep^{-1}\circ y_\gep}}\sigma(x_\gep)\big)\nabla \phi(y_\gep,t_\gep)\\
&=\;^t\sigma(2x_\gep-y_\gep)\nabla \phi(y_\gep,t_\gep)
\longrightarrow X\phi(\hat x,\hat t)=0,\quad\mbox{as }\gep\to0,\end{array}$$
since $(x_\gep)_h-(x_\gep^{-1}\circ y_\gep)_h=(x_\gep\circ y_\gep^{-1}\circ x_\gep)_h=(2x_\gep-y_\gep)_h$, and
$$\begin{array}{ll}
X^2\varphi(x_\gep,t_\gep)&=\;^t\sigma(x_\gep)\;^tD{\tilde\tau_{x_\gep^{-1}\circ y_\gep}}D^2\phi(\tilde\tau_{x_\gep^{-1}\circ y_\gep}(x_\gep),t_\gep)D{\tilde\tau_{x_\gep^{-1}\circ y_\gep}}\sigma(x_\gep)\\
&=\;^t\sigma(2x_\gep-y_\gep)D^2\phi(y_\gep,t_\gep)\sigma(2x_\gep-y_\gep)
\longrightarrow X^2\phi(\hat x,\hat t)\neq0,\quad\mbox{as }\gep\to0.\end{array}$$
Moreover we show that $X\varphi(x_\gep,t_\gep)\neq0$. In fact, as $u(x_\gep,t)-\frac1\gep d_G^4(x_\gep,y)-\phi(y,t)$ has a maximum at $(y,t)=(y_\gep,t_\gep)$,
$$\begin{array}{ll}
X\varphi(x_\gep,t_\gep)&\displaystyle=\;^t\sigma(2x_\gep-y_\gep)\nabla \phi(y_\gep,t_\gep)=-\gep^{-1}\;^t\sigma(2x_\gep-y_\gep)\nabla_yd_{G}^4(x_\gep,y_\gep)\\
&\displaystyle=-\gep^{-1}\;^t\sigma(2x_\gep-y_\gep)\;^tD{\tau_{x_\gep^{-1}}}\nabla N(x_\gep^{-1}\circ y_\gep)
=-\gep^{-1}\;^t\sigma(x_\gep-y_\gep)\nabla N(x_\gep^{-1}\circ y_\gep)\\
&\displaystyle=\gep^{-1}\;^t\sigma(x_\gep-y_\gep)\nabla N(y_\gep^{-1}\circ x_\gep)
=\gep^{-1}XN(y_\gep^{-1}\circ x_\gep).
\end{array}$$
By the previous Lemma \ref{lemma:property norm in Carnot group2} this is null if and only if $(y_\gep)_h=(x_\gep)_h$ and we already know that this cannot be true. Thus by (\ref{eqsubsol}) it holds
$$\frac{\partial\varphi}{\partial t}(x_\gep,t_\gep)+G(x_\gep,t_\gep,\nabla \varphi(x_\gep,t_\gep),D^2\varphi(x_\gep,t_\gep))\leq0$$
and we conclude by letting $\gep\to0$,
$$\begin{array}{ll}
0&\displaystyle\geq\liminf_{\gep\to0}\Big(\frac{\partial\varphi}{\partial t}(x_\gep,t_\gep)+G(x_\gep,t_\gep,\nabla\varphi(x_\gep,t_\gep),D^2\varphi(x_\gep,t_\gep))\Big)\\
&\displaystyle\geq\partial_t\phi(\hat x,\hat t)+G_*(\hat x,\hat t,\nabla\phi(\hat x,\hat t),D^2\phi(\hat x,\hat t)).\end{array}$$
\end{proof}

\begin{rem}\label{corol:BG in Carnot group2} By a remark during the previous proof,
it is not restrictive to assume in Definition \ref{defviscsol} that, if $u$ (respectively $v$) is an upper semicontinuous subsolution (respectively a lower semicontinuous supersolution) of equation (\ref{eqlevelset}) and $\varphi\in C^2(\RRn\times(0,+\infty))$ is a test function for $u$ (resp. for $v$) at the point $(x,t)$, then at any point $(y,s)$ in a neighborhood of $(x,t)$ such that
$$X\varphi(y,s)=X\varphi(x,t)=0,$$
it holds
$$X^2\varphi(y,s)=0.$$
\end{rem}
Complementing Theorem \ref{thmbg} and Remark \ref{remtech}, we obtain the following consequence. It shows, in particular that the notion of solution for the horizontal mean curvature flow equation used in \cite{caci} or in \cite{didr}, which is different from viscosity solutions at characteristic points, is in fact equivalent to standard viscosity solutions and ours.
\begin{cor}\label{corequiv}
Let $u:\R^n\times(0,+\infty)\to\R$ be an upper (respectively lower) semicontinuous function. Function $u$ is a viscosity subsolution  (resp. supersolution) of (\ref{eqlevelset}) if and only if for any $\phi\in C^2(\RRn\times(0,+\infty))$, if $(x,t)\in\RRn\times(0,+\infty)$ is a local maximum (respectively minimum) point for $u-\phi$, one has
\begin{equation}\label{condcaci}
{\partial_t\phi(x,t)}+F_*(x,t,X\phi(x,t),X^2\phi(x,t))
\leq0\end{equation}
(resp.
$${\partial_t\phi(x,t)}+F^*(x,t,X\phi(x,t),X^2\phi(x,t)
\geq0.)$$
\end{cor}
\begin{proof}
Suppose that $(x,t)\in\RRn\times(0,+\infty)$ is a local maximum point for $u-\phi$. If $\nabla u(x,t)\neq0$ then $G(x,t,\nabla u(x,t),D^2u(x,t))=F(x,t,Xu(x,t),X^2(x,t))$ so there is nothing to prove.
We therefore limit ourselves to discuss the case $Xu(x,t)=0$.

If $u$ is a viscosity subsolution, by Remark \ref{remtech} we know that $F_*\leq G_*$, therefore (\ref{condcaci}) is satisfied.

If instead we suppose that (\ref{condcaci}) holds true, then by Theorem \ref{thmbg} we limit ourselves to test functions $\phi$ that satisfy: $X\phi(x,t)=0$ implies $X^2\phi(x,t)=\mathbb O$. In this case $F_*(x,t,X\phi(x,t),X^2\phi(x,t))=0$ and then $\partial_t \phi(x,t)\leq0$. Thus by Theorem \ref{thmbg} we know that $u$ is a viscosity subsolution.
\end{proof}
\begin{rem}
In \cite{baci} the authors require a subsolution $u$ of the horizontal mean curvature flow equation to satisfy
$$\partial_t\phi(x,t)-\mbox{Tr}X^2\phi(x,t)\leq0,$$
if $u-\phi$ has a maximum at $(x,t)$ and $X\phi(x,t)=0$.
If in particuar $\phi$ is in the class of test functions such that $X\phi(x,t)=0$ implies $X^2\phi(x,t)=0$, then $\partial_t\phi(x,t)\leq0$. Therefore $u$ is a subsolution in the sense of Theorem \ref{thmbg} and then it is a viscosity subsolution of (\ref{eqlevelset}).
\end{rem}

\section{Examples of explicit super or subsolutions}

In this section we present examples of super and subsolutions of the geometric equation in the case of the horizontal mean curvature flow equation (mcfe) when $F$ is given in (\ref{mean curvature hamiltonian dim m}).
From Theorem \ref{thmbg} we see that when we deal with functions with separated variables like $u(x,t)=\phi(t)+U(x)$ it it easy to check the (mcfe) at singular points of the operator. If $u-\varphi$ has a maximum/minimum at $(x_o,t_{o})$ and $X\varphi(x_o,t_{o})=0$, then we only need to look at the sign of $\varphi_t(x_o,t_{o})$ provided suitable test functions exist, i.e. $X^2\varphi(x_o,t_o)=\mathbb{O}$, otherwise we have nothing to check.
We start with a general result in step two Carnot groups, based on the definition of convex functions in the group. The definition of $v-$convex function (as in viscosity-convex) is given in Bardi-Dragoni \cite{badr}, where it is discussed and characterised, and the reader can find explicit examples.
\begin{defin}
A continuous function $U:\R^n\to\R$ is $v-$convex in the Carnot group if there is $\alpha\geq0$, and for all test functions $\phi\in C^2$ such that $U-\phi$ has a maximum at $x_o$, then $X^2\phi(x_o,t_o)\geq\alpha I$. If $\alpha>0$ we say that $U$ is strictly v-convex.
\end{defin}
The idea is to build supersolutions of the (mcfe) from a $v-$convex function.
\begin{prop}\label{propsupsol}
In a Carnot group of step 2, let $U\in C(\Omega)$ be continuous and a strictly $v-$convex function. Then for $c,r\in\R$, the function $u(x,t)=ct-U(x)+r$ is a supersolution of (mcfe) for all $c\geq-(m-1)\alpha$, $r\in\R$. 
Suppose moreover that $U$ is nonnegative. Then if $c=-(m-1)\alpha$ and $r>0$, the initial front $\{x:u(x,0)=0\}=\{x:U(x)=r\}$ becomes extinct before time $\bar t= r/({(m-1)\alpha})$.
\end{prop}
\begin{proof}
In order to check the supersolution condition, we use the alternative definition as in Theorem \ref{thmbg}. Let $\varphi\in C^2(\R^n\times(0,+\infty))$ be such that $u-\varphi$ has a minimum at $(x_o,t_o)$. Since $U-(-\varphi(\cdot,t_o))$ has a maximum at $x_o$ and $U$ is strictly $v-$convex, then $-X^2\varphi(x_o,t_o)\geq \alpha I$ for some $\alpha>0$. Therefore it cannot be $X\varphi(x_o,t_o)=0$ if $\varphi$ is an appropriate test function, and then
$$\partial_t\varphi(x_o,t_0)-\hbox{tr }\left(\left(I-\frac{X\varphi(x_o,t_o)}{|X\varphi(x_o,t_o)|}\otimes\frac{X\varphi(x_o,t_o)}{|X\varphi(x_o,t_o)|}\right)X^2\varphi(x_o,t_o)\right)\geq c+(m-1)\alpha\geq0,$$
provided $c\geq-(m-1)\alpha$.

The zero sublevel set of the supersolution $u$ becomes a barrier if a comparison principle holds. At time $t$, if $c=-(m-1)\alpha$,  it is given by $\{x:u(x,t)\geq 0\}=\{x:U(x)\leq-(m-1)\alpha t+r\}$ and becomes empty if $t> r/({(m-1)\alpha})$.
\end{proof}
In the previous proposition, the front may have characteristic points, as we see in some more explicit examples below.

To simplify, we now specialise to Heisenberg like groups.
Building supersolutions seems to be easier than subsolutions in particular if characteristic points are present.
Below we consider as reference space $\R^n=\R^{m}\times\R\ni x=(x_h,x_v)$, $m\geq2$ and suppose that $\sigma(x_h,x_v)=\;^t(I_{ m},Bx_h)$, where $^tB=-B=B^{-1}$ is an $m\times m$ matrix. Notice that then $Bx_h\cdot x_h=0$, $B^2=-I_{ m}$ and $|Bx_h|=|x_h|$.

\begin{exa}
In the first example we avoid characteristic points. For $c,
r\in\R$, consider the family of functions $w(x,t)=ct-|x_h|^2+r$. We easily get that
$$Xw(x,t)=-2x_h,\quad X^2w(x,t)=-2I_{m\times m}.
$$
In particular $|x_h|^2$ is strictly $v-$convex, we can compute exactly the operator
$$\begin{array}{cc}
w_t(x,t)-\hbox{tr}\left(X^2w(x,t)-\frac{X^2wXw\otimes Xw(x,t)}{|Xw(x,t)|^2}\right)=
c+2(m-1)
\end{array}$$
and thus by Theorem \ref{thmbg} and Proposition \ref{propsupsol}, $w$ is a supersolution for $c\geq-2(m-1)$ and a subsolution for $c\leq -2(m-1)$ in $\R^n\times(0,+\infty)$, so $w$ is a viscosity solution, for $c=-2(m-1)$.
Notice that for $r>0$ the zero level set of $w$ is a cylinder with axis $\{x:x_h=0\}$ and it goes extinct at time $t=r/(2(m-1))$.
\end{exa}
In general it is not as easy to find explicit solutions.
\begin{exa}
We consider a function built on the gauge function of the Heisenberg group, namely a variation of the homogeneous norm
$$u(x,t)=ct-G(x_h,x_v)+r,\quad \hbox{where }G(x_h,x_v)=|x_h|^4+4|x_v|^2,$$
and $c,r$ are constants to be decided later.
Notice that the zero level set of $u$ is (we will always regard $r>0$ for convenience)
$$\{(x,t):u(x,t)=0\}=\{(x,t):G(x,t)=r+ct\},$$
therefore it is the boundary of a ball for the distance $G^{1/4}$ centred at the origin. It has characteristic points, namely points where $XG(x,t)=0$ precisely in its intersection with the axis $x_h=0$, as we readily see below.
We can easily compute (here we will do complete calculations and not only the signature of $X^2G$ because we also want to check the subsolution condition)
$$XG(x,t)=\sigma(x) \;^t\nabla G(x,t)=4|x_h|^2x_h+8x_v \;^t(Bx_h),\quad |XG(x,t)|^2=16|x_h|^2G(x,t),$$
$$\begin{array}{cc}
X^2G(x,t)=\;^t\sigma\; D^2G\;\sigma(x)=(I_{m\times m},Bx_h)\left(\begin{array}{cc}8x_h\otimes x_h+4|x_h|^2I_{m\times m}\quad&0\\0&8
\end{array}\right)\;^t(I_{m\times m},\;Bx_h)\\
=8x_h\otimes x_h+4|x_h|^2I_{m\times m}+8Bx_h\otimes Bx_h\geq0.
\end{array}$$
Therefore $G$ is v-convex but not strictly v-convex.
Finally
$$XG\cdot X^2G(x,t)XG(x,t)=(48|x_h|^4x_h+96x_v|x_h|^2\;Bx_h)\cdot XG(x,t)=192|x_h|^4G(x,t).
$$
and since $u$ is smooth, we conclude that, for $x_h\neq0$,
$$\begin{array}{cc}
u_t(x,t)-\hbox{tr}\left(X^2u(x,t)-\frac{X^2uXu\otimes Xu(x,t)}{|Xu(x,t)|^2}\right)=
c+\hbox{tr}\left(X^2G(x,t)-\frac{X^2GXG\otimes XG(x,t)}{|XG(x,t)|^2}\right)\\
=c+(8+4m+8)|x_h|^2	-12|x_h|^2=c+4n|x_h|^2.
\end{array}$$ 
Now we can use our alternative definition to obtain that the viscosity super/subsolution condition is satisfied also at points where the horizontal gradient vanishes.
We conclude that:
\begin{itemize}
\item[(i)]{for $c\geq0$, $u$ is a global supersolution in $\R^n\times\R_+$, since $u_t\geq0$; }
\item[(ii)]{for $c<0$, $u$ is a subsolution but only in the cylinder $\{x:|x_h|< \sqrt{\frac{-c}{4n}}\}$ around the axis $x_h=0$. }
\end{itemize}
Compared to the previous example, now $u$ may itself be a test function and then we cannot fulfil a super or subsolution condition just by the lack of test functions.

Notice that the level sets of $u$ in the supersolution case, which are propagating (super)fronts, have radius nondecreasing in time and it may even be stationary for $c=0$. Instead the radius decreases in time in the subsolution case where however the diameter of the section of the domain of the subsolution vanishes with $c$. The zero level set at time $t=0$ is contained in the cylinder in (ii) provided $-c>4n\sqrt{r}$ and it goes extinct at time $t=-r/c$.
\end{exa}
\begin{exa}
Similar calculations of the previous example can be made for $w(x,t)=ct-|x|^2+r$ we get
$$Xw(x,t)=-(2x_h+2x_vBx_h),\quad X^2w(x,t)=-2I_{m\times m},\quad |Xw(x,t)|^2=-4|x_h|^2(1+x_v^2)
$$
and therefore
$$\begin{array}{cc}
w_t(x,t)-\hbox{tr}\left(X^2w(x,t)-\frac{X^2wXw\otimes Xw(x,t)}{|Xw(x,t)|^2}\right)=
c+2(m-1)+2\frac{|x_h|^2}{1+x_v^2}.
\end{array}$$
Again by Theorem \ref{thmbg}, Proposition \ref{propsupsol} and since $|x|^2$ is strictly v-convex, $w$ is a supersolution in $\R^n\times(0,+\infty)$, for $c\geq-2(m-1)$, and a subsolution for $c<-2(m-1)$
 in the open sets $\{x\in\R^n:|x_h|^2<\gep(1+x_v^2)\}$ if $\gep$ is sufficiently small. 
\end{exa}

We now construct a modification of the second example to build a global subsolution of the mean curvature flow equation whose level sets have characteristic points. We first prove a lemma on change of variables for the horizontal mean curvature operator.
\begin{lem} Let $U\in C^2(\R^n)$ and $\psi:\R\to\R$ be smooth with $\psi'>0$. Then for $W=\psi(U)$, if $XU\neq0$, we have
$$-\hbox{tr }\left(X^2W-\frac{X^2W\;XW\otimes XW(x)}{|XW(x)|^2}
\right)=-\psi'(U)\hbox{tr }\left(X^2U-\frac{X^2U\;XU\otimes XU(x)}{|XU(x)|^2}
\right).$$
\end{lem}
\begin{proof} It is just a matter of computing terms. We obtain
$$\nabla W=\psi'(U)\nabla U,\quad XW(x)=\psi'(U)\;XU(x),\quad |XW(x)|^2=(\psi'(U))^2\;|XU(x)|^2$$
$$\begin{array}{ll}D^2W(x)&=\psi''(U)\;\nabla U\otimes \nabla U(x)+\psi'(U)\;D^2U(x), \\ X^2W(x)&=
\psi''(U)\;X U\otimes X U(x)+\psi'(U)\;X^2U(x)
\end{array}$$
$$\begin{array}{ll}
\hbox{tr }X^2W(x)&=\psi''(U)\;|X U|^2+\psi'(U)\;\hbox{tr }X^2U(x),\\
X^2W\;XW\cdot XW(x)&=(\psi'(U))^2(\psi''(U)\;|XU(x)|^4+\psi'(U)\;X^2U\;XU\cdot XU(x)).
\end{array}$$
Finally, putting things together the first two terms in the previous equations cancel out.
\end{proof}
\begin{exa}
In this example we consider the function
$$v(x,t)=ct-G(x_h,x_v)^{1/2}+r$$
which now is not differentiable at points $(0,t)\in\R^n\times(0,+\infty)$.
However $v$ is locally Lipschitz continuous and is differentiable in the group of variables $x_h$. Moreover there is no smooth test function such that $v-\phi$ has a local minimum at $(0,t)$, and if $v-\phi$ has a local maximum at $(0,t)$, then $\phi_t(0,t)=c$ and $\nabla_{x_h}\phi(0,t)=0$ so that $X\phi(0,t)=0$ since $\sigma(0)=(I_{m\times m},0)$. Therefore to check the mean curvature flow equation at such points we only need to look at the sign of $c$ by Theorem \ref{thmbg}.

We now proceed at points such that $x_h\neq0$. We use the lemma with $\psi(s)=s^{1/2}$, so that $\psi'(s)=1/(2\psi(s))$ and the calculations of the previous example.
Again the zero level sets of $v$ are
$$\{(x,t):v(x,t)=0\}=\{(x,t):G(x,t)=(r+ct)^2\},$$
and we check the equation at non characteristic points.
We obtain, by the lemma,
$$\begin{array}{cc}
v_t(x,t)-\hbox{tr}\left(X^2v(x,t)-\frac{X^2vXv\otimes Xv(x,t)}{|Xv(x,t)|^2}\right)=
c+\frac1{2G(x,t)^{1/2}}4n|x_h|^2.
\end{array}$$
We conclude that $v$ is a supersolution for $c\geq0$ as before, but now, since
${|x_h|^2}\leq{G(x,t)^{1/2}}$,
$v$ becomes a global subsolution for $c\leq-2n$. If $c=-2n$, the extinction time of the zero level set of the subsolution is $t=r/(2n)$. Finally notice that all functions of the family share the same initial condition at time $t=0$ independently of $c$.
\end{exa}

\section{A geometric definition of generalised flow in Carnot groups}

In this section we extend the definition of generalised super and subflows introduced by Barles-Souganidis \cite{basou}, later revisited by Barles and Da Lio \cite{badl}, to the setting of level set equations in Carnot groups, also in view of the ideas described in Section 3. This more geometric definition turns out to determine uniquely the geometric flow of a hypersurface if the usual level set equation determines a unique evolution with empty interior (no fattening). This definition has been proven to be much more efficient when dealing with singularly perturbed problems that give rise to geometric flows and we will use it in \cite{deso5} to extend to the Carnot group setting the classical Allen-Cahn approach.
In the following definition we follow \cite{badl} with one modification, see Remark \ref{remflow}.

\begin{defin}\label{generalized flow in Carnot group def} Let $F:\R^n\times(0,+\infty)\times\R^m\backslash\{0\}\times\mathcal S^m$ be locally bounded and satisfying (F1-2-3), and let $G$ be defined as in (\ref{eqopg}).
A family $\open$ (resp. $\close$) of open (resp. close) subsets of $\RRn$ is called a \emph{generalized superflow} (resp. \emph{subflow}) with normal velocity $-F$ if, for any $x_0\in\RRn$, $t\in(0,T)$, $r>0$, $h>0$ so that $t+h<T$ and for any smooth function $\phi:B(x_0,r]\times[t,t+h]\rightarrow\RR$ such that:
\begin{description}
  \item[(i)] $\partial_t\phi(x,s)+G^*(x,t,\nabla \phi(x,s),D^2\phi(x,s))<0$ in $B(x_0,r]\times[t,t+h]$\\ (resp. $\partial_t\phi(x,s)+G_*(x,t,\nabla \phi(x,s),D^2\phi(x,s))>0$ in $B(x_0,r]\times[t,t+h]$),
  \item[(ii)] for any $s\in[t,t+h]$, $\{x\in B(x_0,r]:\phi(x,s)=0\}\neq\emptyset$ and
  $$|\nabla \phi(x,s)|\neq0\mbox{ on }\{(x,s)\in B(x_0,r]\times[t,t+h]:\phi(x,s)=0\},$$
  \item[(iii)] if there exists a pair $(x,s)\in B(x_0,r]\times[t,t+h]$ so that $\abs{X\phi(x,s)}=0$, then it holds also $\abs{X^2\phi(x,s)}=0$,
  \item[(iv)] $\{x\in B(x_0,r]:\phi(x,t)\geq0\}\subset\Omega_t$ (resp. $\{x\in B(x_0,r]:\phi(x,t)\leq0\}\subset\mathcal{F}_t^c$),
  \item[(v)] for all $s\in[t,t+h]$, $\{x\in \partial B(x_0,r]:\phi(x,s)\geq0\}\subset\Omega_s$ (resp. $\{x\in \partial B(x_0,r]:\phi(x,s)\leq0\}\subset\mathcal{F}_s^c$),
\end{description}
then we have
$$\{x\in B(x_0,r]:\phi(x,s)>0\}\subset\Omega_s,\quad
(\mbox{resp. }\{x\in B(x_0,r]:\phi(x,s)<0\}\subset\mathcal{F}_s^c,)$$
for every $s\in(t,t+h)$.

A family $\open$ of open subsets of $\RRn$ is called a \emph{generalized flow} with normal velocity $-F$ if $\open$ is a superflow and  $(\overline{\Omega}_t)_{t\in(0,T)}$ is a subflow.
\end{defin}
\begin{rem}\label{remflow}
The previous definition focuses on evolution of sets directly instead of looking at the level sets of the solutions of a differential equation. It does this by assuming local comparison with smooth evolutions. Indeed when checking if a collection of open sets provides a superflow, (i) requires the smooth function $\phi$ to be a local strict subsolution, (ii) assumes that the zero level set of $\phi$ is smooth, (iv)-(v) require compatible initial and boundary conditions in the local cylinder between the family of sets and the smooth evolution. The condition (iii) is new and we add it to restrict the family of test functions in view of what we did in Section 3.
As we will see from the proof of the characterisation Theorem \ref{flow-viscositysol Carnot group} below, in view of our Theorem \ref{thmbg}, the condition (iii) can be present or not, the corresponding definition would be equivalent.

It follows immediately by Definition \ref{generalized flow in Carnot group def} that a family $\open$ of open subsets of $\RRn$ is a generalised superflow with normal velocity $-F$ if and only if $(\Omega_t^c)_{t\in(0,T)}$ is a generalised subflow with normal velocity $F$.
\end{rem}

We now state and prove the following result which describes the connection between generalised flows and solutions of (\ref{eqlevelset}).

\begin{thm} \label{flow-viscositysol Carnot group}
\begin{description}
\item[(i)]Let $\open$ be a family of open subsets of $\RRn$ such that the set $\Omega:=\bigcup_{t\in(0,T)}\Omega_t\times\{t\}$ is open in $\RRn\times[0,T]$. Then $\open$ is a generalised superflow with normal velocity $-F$  if and only if the function $\chi=\mathds{1}_{\Omega}-\mathds{1}_{\Omega^c}$ is a viscosity supersolution of (\ref{eqlevelset}).
\item[(ii)]Let $\close$ be a family of closed subsets of $\RRn$ such that the set $\mathcal{F}:=\bigcup_{t\in(0,T)}\mathcal{F}_t\times\{t\}$ is closed in $\RRn\times[0,T]$. Then $\close$ is a generalised subflow with normal velocity $-F$ if and only if the function $\overline\chi=\mathds{1}_{\mathcal{F}}-\mathds{1}_{\mathcal{F}^c}$ is a viscosity subsolution of (\ref{eqlevelset}).
\end{description}
\end{thm}
\begin{proof} We adapt to our situation some of the ideas in \cite{badl} and only consider (i) as the other case is similar.
We first assume that  $\chi=\mathds{1}_{\Omega}-\mathds{1}_{\Omega^c}$ is a supersolution of (\ref{eqlevelset}) and we show that $\open$ is a generalised superflow. To do this we consider a smooth function $\phi$, a point $(x_0,t)\in\RRn\times(0,T)$ and $r,h>0$ satisfying conditions (i--v) in Definition \ref{generalized flow in Carnot group def}. We assume that $\phi\leq1$  in $B(x_0,r]\times[t,t+h]$ (otherwise we change $\phi$ with $\eta\phi$ for $\eta>0$ small enough and we use the homogeneity of $F$).  We consider
$$m:=\min\{\chi(x,s)-\phi(x,s):(x,s)\in B(x_0,r]\times[t,t+h]\}.$$
Since $\phi$ satisfies condition (i), $\chi$ is a supersolution of equation (\ref{eqlevelset}) in $B(x_0,r)\times(t,t+h)$ and it is well known, see e.g. \cite{bcd}, that $\chi$ is therefore also a supersolution in $B(x_0,r)\times(t,t+h]$, we deduce that the minimum $m$ has to be attained either in $\partial B(x_0,r)$ or at time $t$.

Let $(x,s)\in(\partial B(x_0,r)\times[t,t+h])\cup (B(x_0,r]\times\{t\})$. If $x\in\Omega_s$, then $\chi(x,s)=1$ and $(\chi-\phi)(x,s)\geq0$ because $\phi\leq1$ in $B(x_0,r]\times[t,t+h]$. If instead $x\not\in\Omega_s$, then $\chi(x,s)=-1$ and, by (iv) and (v), $(\chi-\phi)(x,s)\geq-1+\delta$ for some $\delta>0$. In any case we can conclude that
$$\chi(y,s)-\phi(y,s)\geq-1+\delta,\quad (y,s)\in B(x_0,r]\times[t,t+h],$$
in particular $\phi(y,s)\leq-\delta$, if $y\notin\Omega_s$. This means that for every $s\in[t,t+h]$,
$$\{y\in B(x_0,r]:\phi(y,s)\geq0\}\cap\Omega_{s}^c=\emptyset,$$
which implies that $\open$ is a generalised superflow with normal velocity $-F$.

Conversely, we assume that $(\Omega_t)_{t\in(0,T)}$ is a generalised superflow and we show that $\chi$ is a supersolution of the equation (\ref{eqlevelset}) in $\RRn\times (0,T)$. We consider a point $(x,t)\in\RRn\times(0,T)$ and a function $\phi\in C^\infty(\RRn\times[0,T])$ so that $(x,t)$ is a strict local minimum point of $\chi-\phi$ and by adding a constant to $\phi$ if necessary we may assume $\phi(x,t)=0$. We want to show that
\begin{equation}\label{thesi flow-viscositysol Carnot group}
\partial_t\phi(x,t)+G^*(x,t,\nabla \phi(x,t),D^2\phi(x,t))\geq0.
\end{equation}
By using the equivalent definition of viscosity solution with a restricted family of test functions, we will suppose that
$X^2\phi(y,s)=0$ whenever $|X\phi(y,s)|=0$.

When $(x,t)$ is in the interior of either $\{\chi=1\}$ or $\{\chi=-1\}$ then $\chi$ is constant in a neighborhood of $(x,t)$ and therefore $\partial_t\phi(x,t)=0$, $\nabla \phi(x,t)=0$ and $D^2\phi(x,t)\leq0$. Since $F$ satisfies (F1-2), then
the inequality in (\ref{thesi flow-viscositysol Carnot group}) is true.
Assume instead that $(x,t)\in\partial\{\chi=1\}\cap\partial\{\chi=-1\}$. Thus, by the lower semicontinuity of $\chi$, $\chi(x,t)=-1$. We suppose by contradiction that there exists an $\alpha>0$ so that we have
$$\partial_t\phi(x,t)+G^*(x,t,\nabla \phi(x,t),D^2\phi(x,t))<-\alpha.$$
We can find $r,h>0$ such that for all $(y,s)\in B(x,r]\times[t-h,t+h]$,
\begin{equation}\label{c1 Carnotgroup}
\partial_t\phi(y,s)+G^*(y,s,\nabla \phi(y,s),D^2\phi(y,s))<-\frac{\alpha}{2}.
\end{equation}
and
\begin{equation}\label{c2 Carnotgroup}
\chi(x,t)-\phi(x,t)=-1<\chi(y,s)-\phi(y,s),\quad(y,s)\neq(x,t).\end{equation}

We consider first the case $|\nabla \phi(x,t)|\neq0$ and 
by choosing smaller $r$, $h$, we assume that $|\nabla \phi|\neq0$ in $B(x,r]\times[t-h,t+h]$.
We introduce the test function $\phi_\delta(y,s):=\phi(y,s)+\delta(s-(t-h))$, for $0<\delta\ll1$. 
Since $\phi(x,t)=0$ and $\nabla \phi(x,t)\neq0$, it is easy to see that if $h$ and $\delta$ are small enough then, for any $t-h\leq s\leq t+h$, the set $\{y\in B(x,r):\phi_\delta(y,s)=0\}$ is not empty. We observe that, for $\delta>0$ small enough, by (\ref{c1 Carnotgroup}) and (\ref{c2 Carnotgroup}), we have
\begin{equation}\label{c3}\phi_\delta(y,s)-1<\chi(y,s),\end{equation}
for all $(y,s)\in (B(x,r)\times\{t-h\})\cup(\partial B(x,r)\times[t-h,t+h])$ and
$$\partial_t\phi_\delta(y,s)+G^*(y,s,\nabla \phi_\delta(y,s),D^2\phi_\delta(y,s))<-\frac{\alpha}{4}$$
for all $(y,s)\in B(x,r]\times[t-h,t+h]$. The inequality (\ref{c3}) implies that
$$\{y\in B(x,r]:\phi_\delta(y,t-h)\geq0\}\subset\Omega_{t-h},\mbox{ and }\{y\in\partial B(x,r):\phi_\delta(y,s)\geq0\}\subset\Omega_{s},$$
for all $s\in[t-h,t+h]$.
Therefore $\phi_\delta$ satisfies (i-ii-iv-v) in Definition \ref{generalized flow in Carnot group def}.  Assumption (iii) holds as well by assumptions on function $\phi$. The definition of superflow then yields
$$\{y\in B(x,r]:\phi_\delta(y,s)>0\}\subset\Omega_s,$$
for every $s\in(t-h,t+h)$. Since $\phi_\delta(x,t)=\delta h>0$, we deduce that $x\in\Omega_t$, and this is a contradiction with $(x,t)\in\partial\{\chi=-1\}$.

Now we turn to the case when ${\nabla \phi(x,t)}=0$. In particular $X\phi(x,t)=0$, $X^2\phi(x,t)=\mathbb O$ by our assumption, and therefore
to prove (\ref{thesi flow-viscositysol Carnot group}), it is then enough to show that
$$\partial_t\phi(x,t)\geq0.$$
We further observe that by the result in \cite{bage} corresponding to our Theorem \ref{thmbg},
we could have restricted $\phi$ to the class of functions such that $\nabla \phi(x,t)=0$ implies
$$\frac{\partial^2\phi}{\partial_{x_i}\partial_{x_j}}(x,t)=\frac{\partial^3\phi}{\partial_{x_i}\partial_{x_j}\partial_{x_k}}(x,t)= \frac{\partial^4\phi}{\partial_{x_i}\partial_{x_j}\partial_{x_k}\partial_{x_l}}(x,t)=0$$
for any $i,j,k,l\in\{1,\dots,n\}$ as we do now. 
Suppose by contradiction that $a:=\partial_t\phi(x,t)<0$. Therefore by Taylor formula
$$\phi(y,s)=\partial_t\phi(x,t)(s-t)+o(\abs{s-t}+\abs{y-x}^4)\quad\mbox{as }s\to t,\;\abs{y-x}\to0.$$
Thus, for all $\gep>0$, there exist $r=r_\gep,h=h_\gep,h'=h'_\gep>0$ such that
$$h'\leq h,\quad h<-\frac{\gep r^4}{a}$$
and, for any $(y,s)\in B(x,r]\times[t-h,t+h']$
$$\begin{array}{rl}
\phi(y,s)&\geq a(s-t)+\frac{a}{2}\abs{s-t}-\gep\abs{y-x}^4\\
&=\frac{a}{2}(s-t)+a(s-t)^+-\gep\abs{y-x}^4
\geq\frac{a}{2}(s-t)-\gep\abs{y-x}^4+ah'.
\end{array}$$
Let $d_G(x,y)=\|x^{-1}\circ y\|_G$, be the distance function defined in (\ref{distance}). For any compact set $K\subset\RRn$, by known results, see e.g. Proposition 5.15.1 in \cite{bolaug}, there exists a positive constant $C_K>0$ so that
$$\frac{\abs{x-y}}{C_K}\leq d_G(x,y)\leq C_K\abs{x-y}^{1/2},$$
for any $x,y\in K$. Thus, if we put $C_r=(C_{B(x,r]})^4$, we get
$$\frac{\abs{x-y}^4}{C_r}\leq N(x^{-1}\circ y)\leq C_r\abs{x-y}^2$$
and by definition of $N$
$$\phi(y,s)\geq\frac{a}{2}(s-t)-\gep C_r N(x^{-1}\circ y)+ah'$$
for any $(y,s)\in B(x,r]\times[t-h,t+h']$.
By (\ref{c2 Carnotgroup}) we can take $\beta>0$ such that
\begin{equation*}2\beta+\phi(y,s)-1<\chi(y,s)\end{equation*}
for all $(y,s)\in (B(x,r]\times\{t-h\})\cup(\partial B(x,r)\times(t-h,t+h'))$. By taking $\beta$ smaller we may also suppose $\beta<\gep r^4/2$. We now proceed similarly as before and consider the function $\psi_\beta(y,s)=(a/2)(s-t)-\gep C_rN(x^{-1}\circ y)+\beta$. Since we can take $h'$ smaller we assume from now on that $h'\leq-\beta/a$. Combining the last two displayed inequalities and the assumptions on $\beta,h,h'$ and $r$ we get
\begin{equation}\label{c4 Carnotgroup}
\psi_\beta(y,s)-1<\chi(y,s)
\end{equation}
for all $(y,s)\in (B(x,r]\times\{t-h\})\cup(\partial B(x,r)\times[t-h,t+h'])$. Thus, with a reasoning similar to the one that we used in the previous case, it is possible to prove that $\psi_\beta$ satisfies conditions (iv) and (v) in Definition \ref{generalized flow in Carnot group def}.  Furthermore we consider  a fixed  $s\in[t-h,t+h']$. We have $\psi_\beta(x,s)=a(s-t)/2+\beta\geq ah'/2+\beta>0$ while for $\abs{y-x}=r$
$$\begin{array}{ll}
\psi_\beta(y,s)&=\frac{a}{2}(s-t)-\gep C_r d_G(x,y)^4+\beta\leq\frac{a}{2}(s-t)-\gep \abs{y-x}^4+\beta\\
&\leq-\frac{ah}{2}-\gep r^4+\beta\leq-\frac{ah+\gep r^4}{2}\leq0.
\end{array}$$
Thus the set $\{y\in B(x,r]:\psi_\beta(y,s)=0\}$ is not empty. Let $y\in B(x,r]$, we compute
$$\nabla \psi_\beta(y,s)=-\gep C_r
\left(\begin{array}{c}
4\abs{y_h-x_h}^2(y_h-x_h)-2\sum_{i=1}^{n-m}(y_{m+i}-x_{m+i}-\langle B^{(i)}x_h,y_h\rangle)B^{(i)}x_h\\
2(y_v-x_v-\langle Bx_h,y_h\rangle).
\end{array}\right)$$
Thus, since the matrices $B^{(i)}$ are skew-symmetric, $\nabla \psi_\beta(y,s)=0$ if and only if $y=x$ and therefore $\abs{\nabla \psi_\beta(y,s)}\neq0$ for every $(y,s)\in\{B(x,r]\times[t-h,t+h']:\psi_\beta(y,s)=0\}$. This proves that $\psi_\beta$ satisfies (ii) in Definition \ref{generalized flow in Carnot group def}. Moreover it satisfies also (iii) since, by Lemma \ref{lemma:property norm in Carnot group2},
$$\abs{X\psi_\beta(y,s)}=0\Leftrightarrow y_h=x_h\Leftrightarrow\abs{X^2\psi_\beta(y,s)}=0.$$
It remains to prove that (i) holds. Since $G^*$ is upper semicontinuous, $G^*(y,s,0,\mathbb O)=0$ and $G$ is geometric, we have that
$$\begin{array}{ll}
{\partial_t\psi_\beta}(y,s)&+G^*(y,s,\nabla \psi_\beta(y,s),D^2\psi_\beta(y,s))\\
&=\frac{a}{2}+G^*(y,s,-\gep C_r\nabla _yN(x^{-1}\circ y),-\gep C_rD^2_{yy}N(x^{-1}\circ y))<0,
\end{array}$$
for $(y,s)\in B(x,r]\times[t-h,t+h']$ and $\gep$ small enough.

Thus, since $\open$ is a generalised superflow, we have
$$\{y\in B(x,r]:\psi_\beta(y,s)>0\}\subset\Omega_s$$
for any $s\in(t-h,t+h')$. But again $\psi_\beta(x,t)=\beta>0$, and this means $x\in\Omega_t$, which is a contradiction.
\end{proof}
\begin{rem}
When we define generalised flows as in Definition \ref{generalized flow in Carnot group def}, then Theorem \ref{flow-viscositysol Carnot group} provides a discontinuous solution of (\ref{eqlevelset}). The discontinuous solution $\chi$ bears a natural initial condition at $t=0$ in the following way. Since $\chi$ is lower semicontinuous, we can extend it at $t=0$ by lower semicontinuity and then define a lower semicontinuous initial condition as
$$\chi_o(x)=\chi_*(x,0).$$
\end{rem}

In order to better understand the nature of Definition \ref{generalized flow in Carnot group def}, we comment briefely on the previous result by recalling the connection between the discontinuous solution of (\ref{eqlevelset}) that appears in Theorem \ref{flow-viscositysol Carnot group} and usual viscosity solutions of (\ref{eqlevelset}), see e.g. Souganidis \cite{soug} and the references therein. Suppose that $u\in C(\R^n\times[0,+\infty)$ is a viscosity solution of (\ref{eqlevelset}). Then we can define the following family of sets, for $t\geq 0$,
\begin{equation}\label{nointerior}
\Gamma_t=\{(x,t):u(x,t)=0\},\quad D_t^+=\{(x,t):u(x,t)>0\},\quad D_t^-=\{(x,t):u(x,t)<0\}.
\end{equation}
The following result is well known in the theory and contains as a consequence of Theorem \ref{flow-viscositysol Carnot group} an existence result for geometric flows. The second part of the statement is based on the validity of a comparison principle, which at the moment for equation (\ref{eqlevelset}) is valid under some restrictions as we discussed in the introduction.
\begin{thm}\label{chi viscosol}
Suppose that $u\in C(\R^n\times[0,+\infty)$ is a viscosity solution of (\ref{eqlevelset}).
With the notation in (\ref{nointerior}), the two functions $\overline\chi(x,t)=\mathds1_{D_t^+\cup\Gamma_t}(x)-\mathds1_{D_t^-}(x),$
$\underline\chi(x,t)=\mathds1_{D_t^+}(x)-\mathds1_{D_t^-\cup\Gamma_t}(x)$ are viscosity solutions of (\ref{eqlevelset}) associated respectively with the discontinuous initial data 
$$\bar w_o=\mathds1_{D_o^+\cup\Gamma_o}-\mathds1_{D_o^-},\quad \underline w_o=\mathds1_{D_o^+}-\mathds1_{D_o^-\cup\Gamma_o},$$ respectively.
In particular the family of sets $(D_t^+)_{t>0}$, $(D_t^+\cup\Gamma_t)^o_{t>0}$ are generalised flows and they coincide if and only if the no-interior condition holds: $\Gamma_t=\partial D^+_t=\partial D^-_t$, for all $t\geq 0$.

If moreover $\Gamma_o$ has an empty interior and a comparison principle holds for the equation (\ref{eqlevelset}), then $\bar \chi,\;\underline\chi$ are respectively the maximal subsolution and the minimal supersolution of the Cauchy problem coupling (\ref{eqlevelset}) with the initial condition  $w_o=\mathds1_{D_o^+}-\mathds1_{D_o^-}$ and it has a unique discontinuous solution if and only if the no-interior condition holds.
The unique solution is given by the function
 \begin{equation}\label{eqchi}
 \chi(x,t)=\mathds1_{D_t^+}(x)-\mathds1_{D_t^-}(x).\end{equation}
\end{thm}


\begin{thebibliography}{99}

\bibitem{bcd}
 Bardi, Martino; Capuzzo-Dolcetta, Italo, Optimal control and viscosity solutions of Hamilton-Jacobi-Bellman equations. With appendices by Maurizio Falcone and Pierpaolo Soravia. Systems \& Control: Foundations \& Applications. Birkh\"auser Boston, Inc., Boston, MA, 1997.

\bibitem{badr}{\sc Bardi, Martino ;  Dragoni, Federica}, \emph{Convexity and semiconvexity along vector fields},
 Calc. Var. Partial Differential Equations  42  (2011),  no. 3-4, 405--427.

\bibitem{badl}  {\sc Barles, Guy; Da Lio, Francesca}, \emph{A geometrical approach to front propagation problems in bounded domains with Neumann-type boundary conditions,} Interfaces Free Bound. 5 (2003), no. 3, 239–274.

\bibitem{bage} {\sc Barles, Guy; Georgelin, Christine}, \emph{A simple proof of convergence for an approximation scheme for computing motions by mean curvature}, SIAM J. Numer. Anal. 32 (1995), no. 2, 484--500.

\bibitem{basou}  {\sc Barles, Guy; Souganidis, Panagiotis,}  \emph{A new approach to front propagation problems: theory and applications,} Arch. Rational Mech. Anal. 141 (1998), no. 3, 237–296.

\bibitem{baci} {\sc Baspinar, E.; Citti, G.}, \emph{Uniqueness of Viscosity Mean Curvature Flow Solution in Two Sub-Riemannian Structures}, SIAM J. Math. Anal. 51 (2019), no. 3, 2633--2659.

\bibitem{bolaug}{\sc Bonfiglioli, A. ;  Lanconelli, E. ;  Uguzzoni, F.},  Stratified Lie groups and potential theory for their sub-Laplacians, Springer Monographs in Mathematics. Springer, Berlin,  2007.

\bibitem{cacima} {\sc Capogna, Luca ;  Citti, Giovanna ;  Manfredini, Maria},  \emph{Regularity of mean curvature flow of graphs on Lie groups free up to step 2}, Nonlinear Anal.  126  (2015), 437--450.

\bibitem{caci}  {\sc Capogna, Luca ;  Citti, Giovanna}, \emph{Generalized mean curvature flow in Carnot groups}, Comm. Partial Differential Equations  34  (2009),  no. 7-9, 937--956.
		
\bibitem{cgg} {\sc Chen, Y.-G., Giga, Y.}, {\sc Goto, S.}, \emph{Uniqueness and existence of viscosity solutions of generalized mean curvature flow equations}, J. Differential Geom. {\bf 33} (1991), 749--786.

\bibitem{cil}{\sc Crandall, M. G., Ishii, H.}, {\sc Lions, P.-L.}, \emph{User's guide to viscosity solutions of second order partial differential equations}, Bull. Amer. Math. Soc. {\bf 27} (1992), 1--67.

\bibitem{dagani} {\sc Danielli, D. ;  Garofalo, N. ;  Nhieu, D. M.}, \emph{ Sub-Riemannian calculus on hypersurfaces in Carnot groups}, Adv. Math.  215  (2007),  no. 1, 292--378.


\bibitem{dzs3} {\sc De Zan, C.; Soravia, P.}, \emph{Geometric flows with discontinuous velocity: a comparison principle},  Int. J. Differ. Equ. (2016), Art. ID 3627896.

\bibitem{deso5} {\sc De Zan, C.; Soravia, P.}, \emph{Singular limits of reaction diffusion equations in degenerate, anisotropic media},  forthcoming.

\bibitem{didr} {\sc Dirr, Nicolas ;  Dragoni, Federica ;  von Renesse, Max}, \emph{Evolution by mean curvature flow in sub-Riemannian geometries: a stochastic approach}, Commun. Pure Appl. Anal.  9  (2010),  no. 2, 307--326.
 
\bibitem{ess} {\sc Evans, L. C., Soner, H. M.}, {\sc Souganidis, P. E.}, \emph{Phase transitions and generalized motion by mean curvature}, Comm. Pure Appl. Math. {\bf 45} (1992), 1097--1123.

\bibitem{es} {\sc Evans, L. C.} , {\sc Spruck, J.}, \emph{Motion of level sets by mean curvature I}, J. Differential Geom. {\bf 33} (1991), 635--681.

\bibitem{flm}  {\sc Ferrari, Fausto; Liu, Qing; Manfredi, Juan J.}, {\emph On the horizontal mean curvature flow for axisymmetric surfaces in the Heisenberg group}, Commun. Contemp. Math. 16 (2014), no. 3, 1350027, 41 pp.

\bibitem{gi} {\sc Giga, Y.},  Surface evolution equations. A level set approach, Monographs in Mathematics, {\bf 99}, Birkh\"auser Verlag, Basel, 2006.

\bibitem{is} {\sc Ishii, H.}, \emph{Hamilton-Jacobi equations with discontinuous Hamiltonians on arbitrary open sets}  Bull. Fac. Sci. Engrg. Chuo Univ.  {\bf 28}  (1985), 33--77.

\bibitem{os} {\sc Osher, S.} , {\sc Sethian, J. A.}, \emph{Fronts propagating with curvature-dependent speed: algorithms based on Hamilton-Jacobi formulations}, J. Comput. Phys. {\bf 79} (1988), 12--49.

\bibitem{soug} {\sc Souganidis, P. E.}, \emph{Front propagation: theory and applications,} Viscosity Solutions and Applications (Montecatini Terme, 1995), I. Capuzzo Dolcetta et al. (eds.), Lecture Notes in Math. {\bf 1660}, Springer Verlag 1997, 186--242.


\end{thebibliography}
\end{document}